\documentclass[12pt,reqno,twoside,fleqn]{amsart}%,a4paper
\usepackage{amscd,amsmath,amssymb,amsthm}%
\usepackage{palatino}
\usepackage[mathcal]{euler}

\numberwithin{equation}{section}
\makeatletter
\renewcommand{\section}{\@startsection {section}{1}{\z@}%
                                   {-3.5ex \@plus -1ex \@minus -.2ex}%
                                   {.5\linespacing}%
                                   {\normalfont\scshape\centering}}
\makeatother

\newtheorem{thm}{Theorem}[section]

\newtheorem{cor}[thm]{Corollary}
\newtheorem{prop}[thm]{Proposition}

\theoremstyle{definition}
\newtheorem{definition}{Definition}[section]

\theoremstyle{remark}
\def\beq#1\eeq{\begin{equation}#1\end{equation}}
\makeatletter
\@ifpackageloaded{euler}{
 
 \newcommand{\onto}{\to\mkern-14mu\to}
 \def\rightarrowfill@#1{\m@th\setboxz@h{$#1\relbar$}\ht\z@\z@
   $#1\copy\z@\mkern-6mu\cleaders
   \hbox{$#1\mkern-2mu\box\z@\mkern-2mu$}\hfill
   \mkern-6mu\mathord\rightarrow$}
 \def\leftarrowfill@#1{\m@th\setboxz@h{$#1\relbar$}\ht\z@\z@
   $#1\mathord\leftarrow\mkern-6mu\cleaders
   \hbox{$#1\mkern-2mu\copy\z@\mkern-2mu$}\hfill
   \mkern-6mu\box\z@$}
 \def\B@R#1#2{\raisebox{-.07ex}{$#1#2$}\mkern-6mu}
 \renewcommand{\hbar}{{\mspace{1mu}\mathpalette\B@R{\mathchar'26}h}}
 \newcommand{\Wedge}{\mathord{\wedge}}
 
}{
 
 \DeclareMathSymbol{\onto}{\mathrel}{AMSa}{"10}
 \renewcommand{\hbar}{{\mathchar'26\mkern-9muh}}
 \newcommand{\Wedge}{\mathsf{\Lambda}}
 
}
\makeatother

\DeclareMathOperator{\sgn}{\mathrm{sgn}}
\newcommand{\C}{\mathcal{C}}
\newcommand{\co}{\mathbb{C}}
\newcommand{\cs}{\mbox{\upshape C}\ensuremath{{}^*}}
\newcommand{\rar}{\rightarrow}
\newcommand{\R}{\mathbb{R}}
\newcommand{\Z}{\mathbb{Z}}

\newcommand{\into}{\hookrightarrow}

\DeclareMathOperator{\rk}{rk}
\newcommand{\inner}{\mathbin{\raise1.5pt\hbox{$\lrcorner$}}}
\DeclareMathOperator{\tr}{tr}

\newcommand{\abs}[1]{\lvert#1\rvert}
\newcommand{\Kahler}{K\"ahler}
\newcommand{\A}{\mathcal{A}}

\newcommand{\Or}{\mathcal O}
\newcommand{\Norm}[1]{\left\|#1\right\|}
\newcommand{\norm}[1]{\lVert#1\rVert}
\DeclareMathOperator{\Mat}{Mat}
\newcommand{\N}{\mathbb N}
\newcommand{\Abs}[1]{\left|#1\right|}

\newcommand{\Li}{\mathcal L}

\DeclareMathOperator{\Tr}{Tr}

\DeclareMathOperator{\ch}{ch}

\newcommand{\deRham}{de\,Rham}

\newcommand{\Lh}{L^2_{\mathrm{hol}}}
\DeclareMathOperator{\curv}{curv}
\DeclareMathOperator{\td}{td}
\DeclareMathOperator{\Pic}{Pic}
\newcommand{\RR}{\mathcal{R}}
\newcommand{\M}{\mathcal{M}}

\begin{document}

\title{Deformation Quantization and Irrational Numbers}
\author{Eli Hawkins} 
\address{Eli Hawkins\\ Department of Mathematics\\ The University of York\\ UK}
\email{mrmuon@mac.com}

\author{Alan Haynes}
\address{Alan Haynes\\ Department of Mathematics\\ The University of York\\ UK}
\curraddr{Department of Mathematics\\ The University of Bristol\\ UK}
\email{alanhaynes@gmail.com}

\subjclass[2010]{Primary 11J70, 46L65; Secondary 11J54, 53D50}

\begin{abstract}
Diophantine approximation is the problem of approximating a real number by rational numbers. We propose a version of this in which the numerators are approximately related to the denominators by a Laurent polynomial. Our definition is motivated by the problem of constructing strict deformation quantizations of symplectic manifolds. We show that this type of approximation exists for any real number and also investigate what happens if the number is rational or a quadratic irrational.
\end{abstract}

\maketitle

\section{Introduction}
Let $\M$ be a  manifold with symplectic form $\omega\in\Omega^2(\M)$.
The starting point of geometric quantization is a complex line bundle $L\to\M$ (with a Hermitian inner product and a compatible connection) whose curvature equals $\omega$.

This can be used to construct a Hilbert space and some correspondence between operators and functions on $\M$. If $\M$ is  the phase space of some classical mechanical system, then these are supposed to be the state space of quantum mechanics and a correspondence between quantum and classical observables. However, the rules for how quantum and classical physics should correspond \cite{ber,haw11,rie11} are stated in terms of the classical limit in which ``Planck's constant'' $\hbar$ approaches $0$.

Changing Planck's constant is equivalent to rescaling the symplectic form, and this can be achieved by taking tensor powers of the line bundle. The curvature of $L^{\otimes k}$ is $k\omega$.

Using these tensor powers and identifying $\hbar=\frac1k$, a strict deformation quantization can be constructed. In particular, the commutator of operators corresponds approximately to $i\hbar$ times the Poisson bracket, which is defined by treating $\omega$ as a matrix and inverting it.

Unfortunately, this procedure isn't always possible. The line bundle $L$ only exists if the symplectic form $\omega$ satisfies an integrality condition --- namely, that the integral of $\omega$ over any closed surface must be an integral multiple of $2\pi$.

What can we do if $\omega$ violates this condition? In particular, what if $\omega$ is not even proportional to an integral form? The solution is to take some more general sequence of line bundles, rather than just tensor powers of a fixed line bundle. The sequence of values of $\hbar$ may also be very different.

The point is that the quantum-classical correspondence only refers to the classical limit, so the curvatures of the line bundles only need to \emph{approximate} multiples of $\omega$.

The difficult part of this is topological. The Chern classes of these line bundles are integral cohomology classes, so they lie on a lattice inside $H^2(\M,\R)$. The condition on the classical limit means that these lattice points must converge toward a given line in $H^2(\M,\R)$, namely, the set of multiples of $[\omega]$.

If $H^2(\M,\R)\cong \R^2$, then this is a matter of approximating a real number by rational numbers --- Diophantine approximation. To construct a strict deformation quantization of $\M$, we need a Diophantine approximation to the ratio between the components of $[\omega]$.

This would be enough to satisfy some definitions of strict deformation quantization, but those definitions do not impose very good behavior in the classical limit. In particular the Jacobi identity for the Poisson bracket is an unnatural and unnecessary condition unless there is some stronger condition on the classical limit. One of us \cite{haw11} has proposed a definition of ``order $N$ strict deformation quantization'' where $2\leq N \leq \infty$. This leads to a stronger condition on the sequence of Chern classes and a more restrictive version of Diophantine approximation.

The purpose of this paper is to study this kind of approximation.

The above motivation was based on the standard construction of geometric quantization, but the modified version of geometric quantization in \cite{haw8} only requires a weaker integrality condition: The integral of $\omega$ over any $S^2\subset\M$ should be a multiple of $2\pi$.

On the other hand, it appears that some sort of integrality condition is necessary from first principles, not just for some constrictions. In \cite{haw11}, one of us proved this for the symplectic $S^2$. In \cite{fed3}, Fedosov proved an integrality condition for ``asymptotic operator representations''.

\subsection{Outline}
We begin in Section~\ref{Rational} by giving a definition for ``order $N$ rational approximation'' to a real number $\alpha\in\R$ and proving that such a thing always exists. 

In Section~\ref{Quantization}, we motivate this definition in two ways from quantization. First, using Proposition~\ref{Toeplitz2}, we show how it arises as a necessary condition in a construction of a deformation quantization. Then, using Theorem~\ref{Asymptotic}, we show that it arises (in the case $N=\infty$) as a necessary condition for the existence of a deformation quantization.

In Section~\ref{Examples}, we examine more precisely what happens if the real number is actually rational or satisfies a quadratic equation with integer coefficients. Finally, in Section~\ref{Conclusions}, we discuss unanswered questions.

\section{Rational Approximation}
\label{Rational}
\subsection{Definition}
Diophantine approximation is one of the oldest topics in number theory. Given a number $\alpha\in\R$, the problem is to approximate $\alpha$ by rational numbers; that is, we need a set of pairs of integers $(r,s)$, such that 
\beq
\label{Diophantine1}
\frac{r}{s} \to \alpha
\eeq
as $s$ increases. This is a rather weak condition, so one usually considers the stronger condition,
\beq
\label{Diophantine2}
r-s\alpha \to 0 .
\eeq
As we shall explain in Section~\ref{Quantization}, the problem of deformation quantization motivates us to define a more restrictive condition:
\begin{definition}
\label{Rational definition}
An \emph{order $N\in\N$ rational approximation} of $\alpha\in\R$ is an infinite subset $\RR\subset \Z^2$, such that there exist real numbers $\gamma_1,\dots,\gamma_N\in\R$ for which
\[
\frac{r}{s} = \alpha + \gamma_1 s^{-1} + \gamma_2 s^{-2} + \dots + \gamma_N s^{-N} + o\left([\abs r + \abs s]^{-N}\right) ,\\
\]
as $\abs r + \abs{s}\to\infty$, for $(r,s)\in\RR$. We will refer to the numbers $s$ as the {\em denominators}. An \emph{infinite order rational approximation} of $\alpha$ is a subset $\RR$ satisfying this condition for any $N$.
\end{definition}
There is nothing special about the expression $\abs r + \abs s$ here. It is simply the easiest norm on $\R^2$ to write down. Any other norm would give an equivalent definition.

It's easy to see that the expansion coefficients $\gamma_1,\dots,\gamma_N$ are uniquely determined by $\RR$.
In the case of infinite order, this is an asymptotic expansion of $r$ as a function of $s$, although $r$ need not actually be a function of $s$.

In terms of this definition, the condition \eqref{Diophantine1} is the definition of an order $0$ rational approximation, and \eqref{Diophantine2} means an order $1$ rational approximation with $\gamma_1=0$.

\subsection{Continued fractions}
In our investigation of finite and infinite order rational approximations we will use continued fractions. Every irrational real number $\alpha$ has a
\emph{simple continued fraction expansion}
\begin{align*}
\alpha = a_0 + \cfrac{1}{a_1+
            \cfrac{1}{a_2+
             \cfrac{1}{a_3+\dotsb}}}=:[a_0; a_1, a_2, a_3, \dots ],
\end{align*}
where $a_0$ is an integer and $a_1, a_2, \dots $ is a sequence of positive integers. The integers $a_0,a_1,\ldots$ are uniquely determined by
$\alpha$ and are called the \emph{partial quotients} in this expansion. The rational numbers
\[
\frac{p_n}{q_n} :=[a_0;a_1,\ldots ,a_n],\quad n\geq 0
\]
are called the \emph{principal convergents} to $\alpha$. We will always assume that $q_n>0$ and $\gcd(p_n,q_n)=1$ for each $n$. Finally for $n\ge 0$ we define the \emph{complete quotients} in the continued fraction expansion of $\alpha$ by
\[
\zeta_n:=[a_n;a_{n+1},\ldots],
\]
and we also define the quantities
\[
\xi_n:=\frac{q_{n-1}}{q_n}.
\]
The most basic facts about continued fractions are that
\beq\label{cffact1}
p_{n+1}=a_{n+1}p_n+p_{n-1},\qquad q_{n+1}=a_{n+1}q_n+q_{n-1},\quad\text{ and}
\eeq
\beq\label{cffact2}
\frac{1}{2q_nq_{n+1}}\le \left|\alpha-\frac{p_n}{q_n}\right|\le\frac{1}{q_nq_{n+1}}.
\eeq
In our applications we will also use the facts that
\beq\label{cffact3}
\alpha-\frac{p_n}{q_n}=\frac{(-1)^n}{q_n^2(\zeta_{n+1}+\xi_n)}\quad\text{ and}
\eeq
\beq\label{cffact4}
\xi_n=[0;a_n,a_{n-1},\ldots ,a_1].
\eeq
Proofs of all of these facts can be found in \cite{RockettSzusz1992}. The following proposition gives a representation of natural numbers in terms of denominators of convergents to $\alpha$. This is known as the \emph{Ostrowski expansion of a natural number} with respect to $\alpha$.
\begin{prop}\label{Ostrowski1}
Suppose $\alpha\in\R$ is irrational. Then for every $s\in\N$ there is a unique integer $M\ge0$ and a unique sequence $\{c_{n+1}\}_{n=0}^\infty$ of integers such that $q_M\le s< q_{M+1}$ and
\beq\label{Ostexp1}
s=\sum_{n=0}^\infty c_{n+1}q_n,
\eeq
with $0\le c_1<a_1$ and $0\le c_{n}\le a_{n}$  for all $n\ge 1$,
\[
c_{n+1}=a_{n+1} \implies c_n=0 ,
\]
and
\[
c_{n+1}=0 \quad \text{for all}  \quad n>M.
\]
\end{prop}
We can construct a similar expansion for real numbers. For $n\ge 0$ let
\beq\label{D_ndef}
D_n:=q_n\alpha-p_n.
\eeq
By (\ref{cffact1}) these quantities satisfy the identities
\beq\label{D_nfact1}
a_{n+1}D_n=D_{n+1}-D_{n-1}\quad\text{for}\quad n\ge1,
\eeq
and it is also not difficult to show that
\beq\label{D_nfact2}
D_n=(-1)^n\|q_n\alpha\| \quad\text{for}\quad n\ge 1,
\eeq
where $\norm{\,\cdot\,}$ denotes the distance to the nearest integer. The following proposition provides us with a way of expanding real numbers in terms of the quantities $D_n$. We will call this the {\em Ostrowski expansion of a real number} with respect to $\alpha$.
\begin{prop}\label{Ostrowski2}
Suppose $\alpha\in [0,1)$ is an irrational number with continued fraction expansion denoted as above. For any $\gamma\in [-\alpha,1-\alpha)$ that satisfies
\beq\label{gam rest}
\|s\alpha-\gamma\|>0~\text{ for all }~s\in\Z
\eeq
there is a unique sequence $\{b_{n+1}\}_{n=0}^{\infty}$ of integers such that
\beq\label{Ostexp2}
\gamma=\sum_{n=0}^\infty b_{n+1}D_n,
\eeq
\[
\text{with }~0\le b_1<a_1,\quad 0\le b_{n+1}\le a_{n+1}\ \text{ for } \ n\ge 1,\qquad\text{and}
\]
\[
b_n=0 \quad \text{whenever}\quad b_{n+1}=a_{n+1}\ \text{ for some }n\ge 1.
\]
\end{prop}
We point out that (\ref{cffact1}), (\ref{cffact2}), and (\ref{D_nfact2}) together imply that the series (\ref{Ostexp2}) is absolutely convergent. Proofs of Propositions \ref{Ostexp1} and \ref{Ostexp2} can be found in \cite[Chapter 3]{RockettSzusz1992}. The reason for our interest in Ostrowski expansions is that they give us a precise and convenient way of working with the quantities $\|s\alpha-\gamma\|,~s\ge 1$, as illustrated by the following proposition.
\begin{prop}\cite[Lemma 5]{berhayvel1}\label{nalph-gam prop}
Let $\alpha\in [0,1)$ be irrational and suppose that $\gamma\in [-\alpha,1-\alpha)$ satisfies (\ref{gam rest}). Choose an integer $s\in\N$ and, referring to the Ostrowski expansions (\ref{Ostexp1}) and (\ref{Ostexp2}), write $\delta_{n+1}:=c_{n+1}-b_{n+1}$ for $n\ge 0$. Let $m$ be the smallest integer for which $\delta_{m+1}\not= 0$. If $m\ge 4$ then
\beq\label{nalph-gam formula1}
\|s\alpha-\gamma\|=\left|\sum_{n=m}^\infty \delta_{n+1}D_n\right|=\sgn(\delta_{m+1}D_m)\cdot\sum_{n=m}^\infty \delta_{n+1}D_n.
\eeq
\end{prop}
Combining this proposition with (\ref{cffact3}) and (\ref{D_nfact2}) gives us the following corollary.
\begin{cor}\label{nalph-gam formula2}
With the same notation as in Proposition \ref{nalph-gam prop}, if $m\ge 4$ then
\[
\|s\alpha-\gamma\|=(-1)^m\sgn(\delta_{m+1})\cdot\sum_{n=m}^\infty \frac{(-1)^n\delta_{n+1}}{q_n(\zeta_{n+1}+\xi_n)}.
\]
\end{cor}
The essence of Proposition \ref{nalph-gam prop} is that when $m\ge 4$ the term $\delta_{m+1}D_m$ dominates the rest of the series in (\ref{nalph-gam formula1}). This can be exploited to give good estimates for the quantities $\|s\alpha-\gamma\|$. For our purposes we only need upper bounds, and the following corollary of Proposition \ref{nalph-gam prop} (proved in \cite{berhayvel1}) will suffice.
\begin{cor}\label{nalph-gam estimate1}
With the same notation as in Proposition \ref{nalph-gam prop}, if $m\ge 4$ then
\[
\|s\alpha-\gamma\|\le (|\delta_{m+1}|+2)\|q_m\alpha\|.
\]
\end{cor}

\subsection{Existence}
Now we return to our problems about finite and infinite order rational approximations. First we show that every irrational number has an infinite order approximation.
\begin{thm}\label{arbitrary approx thm}
Let $\Psi:\N\rar\R^+$ be a decreasing function, and suppose that $\alpha\in [0,1)$ is irrational.
There exists a real number $\gamma$ and a strictly increasing sequence $\{s_k\}_{k=1}^\infty$ such that
\[\|s_k\alpha-\gamma\|\le \Psi(s_k)~\text{ for }~k\ge 1.\]
\end{thm}
\begin{proof}
First we construct a sequence $\{n_k\}$ of positive integers by setting $n_1=4$ and then, for $k\ge 1$, choosing $n_{k+1}$ to be the smallest integer greater than $n_k+1$ for which
\[
\frac{3}{q_{n_{k+1}}}\le \Psi(q_{n_k+1}).
\]
Let $\gamma\in\R$ be the real number with Ostrowski expansion, in terms of $\alpha$, given by
\begin{align*}
&b_{n_k+1}=1\text{ for all }~k\in\N,~\text{ and}\\
&b_{n+1}=0~\text{for all}~n\in\N\smallsetminus\{n_k\},
\end{align*}
and for each $k$ let $s_k$ be defined by
\begin{align*}
s_k&=\sum_{n=0}^{n_k}b_{n+1}q_n=\sum_{m=1}^k q_{n_m}.
\end{align*}
Then by Corollary \ref{nalph-gam estimate1} and inequality \eqref{cffact2} we have that
\begin{align*}
\|s_k\alpha-\gamma\|&\le 3\|q_{n_{k+1}}\alpha\|\le\frac{3}{q_{n_{k+1}+1}}.
\end{align*}
Since $s_k\le q_{n_k+1}$ and $\Psi$ is decreasing, the right hand side here is less than $\Psi(s_k)$.
\end{proof}
For example, by choosing $\Psi(s)=e^{-s}$ we obtain the following corollary.
\begin{cor}
\label{Existence}
If $\alpha\in\R$ is irrational, then there exists an infinite order rational approximation $\RR$ to $\alpha$ with $\gamma_j=0$ for all $j\geq2$. That is, there exists a real number $\gamma_1$ such that, for all $N\in\N$,
\[
\frac{r}{s} = \alpha + \gamma_1 s^{-1} + o\left([\abs r + \abs s]^{-N}\right) ,\quad\text{as}\quad|s|\rar\infty,~(r,s)\in\RR.
\]
\end{cor}
\begin{proof}
With $s_k$ defined by Theorem~\ref{arbitrary approx thm}, let $r_k$  be the nearest integer to $\alpha s_k$. The rational approximation is then 
\[
\RR = \{(r_k,s_k)\mid k\in\N\} .
\]
\end{proof}
The proof of Theorem \ref{arbitrary approx thm} tells us how to construct infinite order approximations to any irrational number. A more subtle problem is to try to construct infinite order approximations where the denominators do not grow too quickly. In Section 6 we will demonstrate a construction for quadratic irrationals which produces infinite order approximations with denominators that grow at most exponentially. By contrast, for the integers $s_k$ constructed in the proof of Theorem \ref{arbitrary approx thm} with $\Psi(s)=e^{-s}$, we have that $s_k\ge e^{e^{\cdots (k\text{-times})}}$.

\section{Quantization}
\label{Quantization}
\subsection{Definition}
The idea of strict deformation quantization was conceived by Rieffel \cite{rie11}. There are several variations on his definition and several ways of describing the structure. We will use a continuous field of \cs-algebras and a quantization map.

This is the definition of quantization given in \cite{haw11}:
\begin{definition}
Let $\A_0$ be a Poisson ${}^*$-subalgebra of functions on a Poisson manifold $\M$, large enough to separate points.
An \emph{order $N$ strict deformation quantization} $(I,A,Q)$ of $\A_0$ consists of:
a locally compact subset $I \subseteq \R$ with $0\in I$ an accumulation point,
a continuous field of \cs-algebras $A$ over $I$,
and a ${}^*$-linear map $Q:\C^\infty_0(\M)\to \Gamma(I,A)$ such that:
\begin{enumerate}
\item At $0\in I$, this map is an inclusion $Q_0:\A_0\into A_0\subset\C_b(\M)$ of $\A_0$ as a dense ${}^*$-subalgebra;
\item for $f,g\in\A_0$, there exist functions $C_1(f,g),\dots,C_N(f,g)\in\A_0$ such that
\[
Q_\hbar(f)Q_\hbar(g) = Q_\hbar(fg+ \hbar C_1(f,g)+\dots + \hbar^N C_N(f,g)) + o(\hbar^N) \; ;
\]
\item for $f,g\in\A_0$,
\[
C_1(f,g) - C_1(g,f) = i\{f,g\} .
\]
\end{enumerate}
\end{definition}
It is easy to check that $C_j(f,g)$ is uniquely determined by $f$ and $g$, which justifies this cumbersome notation.

In practice, the given structure is usually the collection of algebras and maps $Q_\hbar : \A_0 \to A_\hbar$ for $\hbar\neq0\in I$. The continuous field structure is then constructed from this.

\subsection{A Construction}
Let $\M$ be a compact, \Kahler\ manifold with symplectic form $\omega \in \Omega^2(\M)$. Suppose that $L\to\M$ is a Hermitian, holomorphic line bundle with curvature, $\curv L = \omega$. The space $\Lh(\M,L)$ of holomorphic sections of $L$ is finite-dimensional, so it is automatically a closed subspace of the Hilbert space $L^2(\M,L)$ of square-integrable sections of $L$ (defined using the Hermitian inner product and the \Kahler\ volume form). Let
\[
\Pi_L : L^2(\M,L) \to \Lh(\M,L)
\]
be the orthogonal projection onto this subspace. There is an obvious representation of the algebra of continuous functions $\C(\M)$ on $L^2(\M,L)$, defined by pointwise multiplication. In particular, if $f\in\C(\M)$ and $\psi\in\Lh(\M,L)$, then the product $f\psi\in L^2(\M,L)$ is square-integrable, so we can construct a vector $\Pi_L(f\psi)\in \Lh(\M,L)$. This construction defines a map,
\begin{gather*}
T_L : \C(\M) \to \Li[\Lh(\M,L)] \\
T_L(f)\psi := \Pi_L(f\psi) .
\end{gather*}

The tensor powers $L^{\otimes k}$ are also positive, Hermitian, holomorphic line bundles, but with the curvature rescaled:
\[
\curv L^{\otimes k} = k \omega .
\]
For any smooth functions, $f$ and $g$, the product $T_{L^{\otimes k}}(f)T_{L^{\otimes k}}(g)$ can be asymptotically expanded in $k$. To be precise:
\begin{thm}
\label{Toeplitz1}
\cite{sch} For any $f,g\in\C^\infty(\M)$, there exist unique functions $C^\omega_j(f,g)\in\C^\infty(\M)$ for $j=0,1,2,\dots$ starting with $C^\omega_0(f,g)=fg$ such that for any $N\in\N$,
\beq
\label{Toeplitz_bound}
T_{L^{\otimes k}}(f)T_{L^{\otimes k}}(g) = \sum_{j=0}^N k^{-j}\, T_{L^{\otimes k}}[C^\omega_j(f,g)] + \Or(k^{-N-1})
\eeq
for all $k\in\N$. Antisymmetrizing $C^\omega_1$ gives,
\[
C^\omega_1(f,g)-C^\omega_1(g,f) = i \{f,g\}_\omega,
\]
the Poisson bracket determined by $\omega$ as a symplectic form. The norm of $T_{L^{\otimes k}}(f)$ converges to,
\beq
\label{norm_limit}
\lim_{k\to\infty}\Norm{T_{L^{\otimes k}}(f)} = \norm{f} .
\eeq
\end{thm}
Equation \eqref{Toeplitz_bound} for $N=0$ and eq.~\eqref{norm_limit} imply that there exists a unique continuous field $A$ of \cs-algebras over
$\{0,\dots,\tfrac13,\tfrac12,1\}\subset \R$ with $A_0:= \C(\M)$ and $A_{1/k} := \Li[\Lh(\M,L^{\otimes k})]$ such that for any $f\in\C^\infty(\M)$,
\[
\begin{split}
Q_0(f)&=f \\
Q_{1/k}(f) &= T_{L^{\otimes k}}(f)
\end{split}
\]
defines a continuous section $\hbar\mapsto Q_\hbar(f)$. The rest of the theorem shows that this is an infinite order quantization of $\C^\infty(\M)$.

\begin{thm}
\label{Continuity}
For each $j$, $C_j^\omega(f,g)$ is a bidifferential operator on $f$ and $g$, determined by $\omega$.
The value of $C_j^\omega(f,g)$ at each point of $\M$ depends continuously upon $\omega$ in the Fr\'echet topology (or in the $\C^m$ topology for some $m$). As a function of $\omega$, it is homogeneous of degree $-j$.
\end{thm}
\begin{proof}
This follows immediately from the results of \cite{k-s}, where it is shown that these are the terms of a ``star product with separation of variables'' (although with the order of multiplication reversed) and that the formal $2$-form classifying this product is simply constructed from $\omega$. This implies that $C^\omega_j$ can (in principle) be constructed from the complex structure, $\omega$, and finitely many of its derivatives. This implies the stated continuity.

The homogeneity is because the deformation parameter $\hbar=1/k$ and symplectic form only enter the formal $2$-form in the combination $\hbar^{-1}\omega$, and the star product can be constructed from the formal $2$-form.
\end{proof}
This implies that eq.~\eqref{Toeplitz_bound} can be stated more directly in terms of $\curv L^{\otimes k} = k \omega$ as,
\beq
\label{Toeplitz_bound2}
T_{L^{\otimes k}}(f)T_{L^{\otimes k}}(g) = \sum_{j=0}^N  T_{L^{\otimes k}}[C^{k\omega}_j(f,g)] + \Or(k^{-N-1}) .
\eeq

This approach of taking tensor powers of a fixed line bundle gives a quantization of $\M$ with the symplectic form $\omega$, which is  by definition a closed, type $(1,1)$ differential form, but it is not arbitrary.  Since the cohomology class $[\frac\omega{2\pi}]=c_1(L) \in H^{1,1}(\M)$ is the first Chern class of $L$, it must be integral.

Constructing a quantization for a symplectic form $\omega$ without this integrality property is more subtle.
Instead of taking tensor powers of a fixed line bundle, we can more generally choose some sequence of holomorphic, Hermitian line bundles $\{L_k\}_{k=1}^\infty$, with positive curvatures $\omega_k:=\curv L_k$. Instead of identifying $\hbar$ with $\frac1k$,  take some non-repeating sequence $\{\hbar_k \in\R\}_{k=1}^{\infty}$ with $\lim_{k\to\infty}\hbar_k=0$. Now define
\[
I := \{0,\hbar_k\mid k\in\N\} \subset \R
\]
$A_0:= \C(\M)$, $A_{\hbar_k} := \Li[\Lh(\M,L_k)]$, and for $f\in\C^\infty(\M)$,
\beq
\label{T_Q}
\begin{split}
Q_0(f) &:= f \\
Q_{\hbar_k}(f) &:= T_{L_k}(f) .
\end{split}
\eeq

The idea is to use this to define a quantization.
If the sequence of curvatures is reasonably well behaved, then the leading order approximation to the commutator will be,
\[
[T_{L_k}(f),T_{L_k}(g)]\approx i T_{L_k}(\{f,g\}_{\omega_k}) .
\]
We want this to be (approximately) $i\hbar_k T_{L_k}(\{f,g\}_\omega)$, therefore we need (for any $f,g\in\C^\infty(\M)$)
\[
\{f,g\}_\omega = \lim_{k\to\infty} \hbar_k^{-1}\{f,g\}_{\omega_k} ,
\]
or equivalently,
\[
\omega = \lim_{k\to\infty}\hbar_k\omega_k ,
\]
where the topology on $\Omega^2(\M)$ is the $\C^0$ topology given by a sup-norm defined with an arbitrary metric.
This means that the \Kahler\ metrics given by these curvatures must --- after rescaling --- converge to the \Kahler\ metric given by $\omega$.

Recall that the  Riemann and Ricci curvatures of a \Kahler\ manifold are determined by the metric and are invariant under rescaling the metric.

\begin{prop}
\label{Toeplitz2}
These maps \eqref{T_Q} define a first order quantization of  $\M$ with the symplectic structure $\omega$, if
\begin{itemize}
\item
in the $\C^0$ topology
\beq
\label{omega_converge}
\omega = \lim_{k\to\infty}\hbar_k\omega_k ,
\eeq
\item
the magnitude of the Riemannian curvature of the \Kahler\ structure $\hbar_k\omega_k$ is of order $o(\hbar_k^{-1})$,
\item
and the magnitude of the derivative of the Ricci curvature is of order $o(\hbar_k^{-2})$.
\end{itemize}
In particular, this is true if \eqref{omega_converge} converges in the $\C^3$ topology.
\end{prop}
\begin{proof}
The calculations in \cite[Lem.~4.8]{haw8} show that for any $f,g\in\C^\infty(\M)$,
\beq
\label{T_1sta}
T_{L_k}(f)T_{L_k}(g) = T_{L_k}\left[fg + iC_1^{\omega_k}(f,g)\right] + o(\hbar_k) ,
\eeq
where $iC_1^{\omega_k}(f,g)$ is the contraction of the holomorphic derivative of $f$ with the antiholomorphic derivative of $g$ using the \Kahler\ metric defined by $\omega_k$. In the notation of \cite{haw8}, $s=\hbar_k^{-1}$, the norms are taken using the rescaled \Kahler\ structure $\hbar_k\omega_k$, $\hat K$ is constructed from the Ricci curvature, and $K_2$ is constructed from the Riemann tensor.

(Alternately, we can take $s=1$. In that case, the norms are taken with respect to the \Kahler\ structure determined by $\omega_k$. This means that the norm of the derivative of $f$ is of order $\Or(\hbar_k)$, the norm of the second derivative is of order $\Or(\hbar_k^2)$, and the norms of the Riemann tensor and its derivative are rescaled by $\hbar_k$ and $\hbar_k^2$, respectively.)

Since we are assuming (eq.~\eqref{omega_converge}) that $\hbar_k\omega_k= \omega + o(1)$, the approximation \eqref{T_1sta} is equivalent to
\beq
\label{T_1st}
T_{L_k}(f)T_{L_k}(g) = T_{L_k}\left[fg + i\hbar_k C_1^{\omega}(f,g)\right] + o(\hbar_k) .
\eeq
In particular,
\beq
\label{T_0th}
T_{L_k}(f)T_{L_k}(g) = T_{L_k}(fg) + \Or(\hbar_k) ,
\eeq
and
\beq
\label{T_Poisson}
[T_{L_k}(f),T_{L_k}(g)] = i\hbar_k T_{L_k}(\{f,g\}_\omega) + o(\hbar_k) .
\eeq
By the reasoning in \cite[Lem.~7.9]{haw8}, eq.~\eqref{T_Poisson} implies that the normalized trace of $T_{L_k}(f)$ converges to the normalized integral of $f$. By the reasoning in \cite[Thm.~7.10]{haw8}, this and eq.~\eqref{T_0th} imply that $Q$ \eqref{T_Q} does define sections of a unique continuous field over $I$. Finally, eq.~\eqref{T_Poisson} is the statement that this is a quantization for the symplectic structure $\omega$, and eq.~\eqref{T_1st} is the statement that this is a first order quantization.

In particular, if $\hbar_k\omega_k$ converges in the $\C^3$ topology, this implies that the Riemann tensor and its derivative converge (and are bounded) in the $\C^0$ topology.
\end{proof}

The question now is how well behaved the sequence $\{\omega_k\}_{k=1}^\infty$ must be to give an order $N$ quantization.

First, with eq.~\eqref{Toeplitz_bound} in mind, let's suppose that the sequence is sufficiently well behaved that
\beq
\label{T_expansion}
T_{L_k}(f)T_{L_k}(g) = \sum_{j=0}^N  T_{L_k}[C^{\omega_k}_j(f,g)] + o(\hbar_k^N) ,
\eeq
where $C^{\omega_k}_j$ is as defined in Theorem~\ref{Continuity}.

We want an expansion of the form,
\[
T_{L_k}(f)T_{L_k}(g) = \sum_{j=0}^N  \hbar_k^jT_{L_k}[C_j(f,g)] + o(\hbar_k^N) .
\]
So, we need to approximate $C^{\omega_k}_j(f,g)$ by a polynomial in $\hbar_k$.
Because of the homogeneity of $C^{\omega_k}_j$, this approximation can be achieved by assuming that $\omega_k$ is approximated by a Laurent polynomial in $\hbar_k$. We have already assumed that $\omega_k\approx \omega \hbar_k^{-1}$, so this must now be corrected with nonnegative powers of $\hbar_k$.

The remaining question is to what order $\omega_k$ needs to be approximated by a Laurent polynomial in order to give an order $N$ quantization.

If $\omega_k$ is approximated to order $o(\hbar_k^n)$ by a Laurent polynomial $\omega \hbar_k^{-1}+\dots$, then $\omega_k^{-1}$ is approximated to order $o(\hbar_k^{n+2})$ by a polynomial $\omega^{-1}\hbar_k+\dots$.

So, to construct an order $N$ quantization of $\M$ with the symplectic form $\omega$, we need a sequence of line bundles $L_k$ and a sequence of numbers $\hbar_k$, such that the curvature of $L_k$ is approximated to order $o(\hbar_k^{N-2})$ in the Fr\'echet topology by a Laurent polynomial $\omega\hbar_k^{-1}+\dots\in \Omega^2(\M)[\hbar_k^{-1},\hbar_k]$.

This is more than enough to satisfy the hypotheses of Proposition~\ref{Toeplitz2}. It seems quite plausible that this is enough to satisfy the assumption \eqref{T_expansion}, but proving that would require generalizing most of the results in the book \cite{b-g}.

Finally, to achieve an infinite order quantization, we must satisfy these conditions for all $N$. This translates to the existence of an asymptotic expansion for $\omega_k$ as a Laurent series in $\hbar_k$.

The main point here is that it is not trivial to find such a sequence of line bundles, because the curvature of a line bundle is not arbitrary, but must determine an integral cohomology class, $c_1(L_k) = [\frac{\omega_k}{2\pi}]$.

This sequence of integral Dolbeault cohomology classes
\[
c_1(L_k) \in H^{1,1}(\M)
\]
has the property that it can be approximated to order $o(\hbar_k^{N-2})$ by a Laurent polynomial in $H^{1,1}(\M)[\hbar_k^{-1},\hbar_k]$ with leading term $[\frac\omega{2\pi}]\hbar_k^{-1}$.

The sequence of numbers $\hbar_k$ doesn't really carry any additional information here. If such a sequence exists, then a valid one can easily be determined from $c_1(L_k)$. It is the ratios between components that are interesting here.

The simplest nontrivial case occurs when $\dim H^2(\M)=2$, so let's consider that case and identify $H^{1,1}(\M)=\R^2$. The integral part of Dolbeault cohomology is identified with $\Z^2\subset\R^2$.

Suppose that $[\frac\omega{2\pi}] = (\alpha,1)$ for some real number $\alpha\in\R$. Denote the Chern classes by $c_1(L_k) = (r_k,s_k)$.

The condition on the second component is
\[
s_k = \hbar_k^{-1}+\dots + o(\hbar_k^{N-2})
\]
which is easily satisfied by choosing $\hbar_k= s_k^{-1}$.

With this choice, the condition on the first component becomes
\[
\begin{split}
r_k &= \alpha \hbar_k^{-1} + \dots + o(\hbar_k^{N-2}) \\
&= \alpha s_k + \gamma_1 + \gamma_2 s_k^{-1} + \dots + \gamma_{N-1} s_k^{-N+1} + o(s_k^{-N+2}) ,
\end{split}
\]
for some real numbers $\gamma_1,\dots,\gamma_{N-1}\in\R$. Equivalently,
\[
\frac{r_k}{s_k} = \alpha + \gamma_1 s_k^{-1} + \dots + \gamma_{N-1} s_k^{-N+1} + o(s_k^{-N+1}) .
\]
In other words, the set of pairs $(r_k,s_k)$ must be an order $N-1$ rational approximation to the real number $\alpha$.

Likewise, for an infinite order quantization, we need an infinite order rational approximation.

\subsection{An Obstruction}
\begin{definition}
A \emph{formal deformation quantization} \cite{bffls,wei} of a Poisson manifold $\M$ is an associative $\co[[\hbar]]$-linear product on the space of formal power series $\C^{\infty}(\M)[[\hbar]]$ of the form
\[
f * g = fg + \sum_{j=1}^\infty \hbar^j C_j(f,g)
\]
where $C_1(f,g)-C_1(g,f) = i\{f,g\}$.
\end{definition}

Suppose that $\A_0\subseteq\C^\infty_b(\M)$ is a Poisson subalgebra of bounded smooth functions, whose restriction to any compact coordinate patch gives all smooth functions there.
Any infinite order strict deformation quantization $(I,A,Q)$ of $\A_0$ determines a formal deformation quantization: For any $f,g\in \A_0$,
\[
Q_\hbar(f) Q_\hbar(g) \sim Q_\hbar(f*g) ,
\]
where $\sim$ means that for any $N\in\N$ if the formal power series on the left is truncated at order $\hbar^N$, then the norm of the difference of the two sides is bounded by a multiple of $\hbar^{N+1}$.

When $\M$ is symplectic, any formal deformation quantization determines \cite{fed2,fed3,n-t1} a characteristic cohomology class $\theta \in \hbar^{-1} H^2(\M)[[\hbar]]$ which is given to leading order by the symplectic form as
\[
\theta = \frac{[\omega]}{2\pi \hbar} + \dots .
\]
(This is related to Fedosov's notation by $\theta = -\frac{\Omega}{2\pi \hbar}$.)
Two formal deformation quantizations determine the same cohomology class if and only if they are isomorphic by an isomorphism that reduces modulo $\hbar$ to the identity on $\C^\infty(\M)$.

Let $n:=\frac12\dim\M$.
Any formal deformation quantization of a symplectic manifold admits a natural $\co[[\hbar]]$-linear trace
\[
\Tr : \C^{\infty}_c(\M)[[\hbar]] \to \hbar^{-n} \co[[\hbar]] .
\]
%This satisfies the trace identity $Tr(f*g)=\Tr(g*f)$. It is unique modulo normalization, that is, any two such traces are proportional by some element of $\co[[\hbar]]$. However, there is a natural choice of normalization.
This natural trace is given to leading order by the symplectic volume form,
\[
\Tr f = \frac1{n!\hbar^n}\int_M f \,\omega^{n}+ \dots .
\]

This trace is the subject of the algebraic index theorem.  Let 
\[
e_0=e_0^2 \in \Mat_m[\C^\infty(\M)]
\]
be an idempotent matrix of smooth functions. Under the $*$-product, it is only approximately idempotent (modulo $\hbar$).
However, this can be corrected to a $*$-product idempotent $e_\hbar$, such that $e_\hbar\equiv e_0 \mod \hbar$.

Suppose, for simplicity, that $\M$ is compact. The trace of the $*$-product naturally extends to matrices, so $\Tr e_\hbar$ is a meaningful expression, and this is what the algebraic index theorem computes. To state it, we need one more definition: Since $e_0$ is an idempotent matrix of functions, it determines a vector subbundle of $\co^m\times\M$, whose fiber at $x\in\M$ is the image $e_0(x) \co^m$; write $\ch e_0$ for the Chern character of this bundle.

\begin{thm}
\label{Algebraic_index}
Let $*$ be any formal deformation quantization of a compact symplectic manifold $\M$, with characteristic class $\theta$. Let $e_0\in\Mat_m[\C^\infty(\M)]$ be any idempotent. For any $*$-idempotent $e_\hbar\equiv e_0 \mod \hbar$, the trace is
\[
\Tr e_\hbar = \int_M \ch e_0 \wedge e^\theta\wedge \hat A(TM) .
\]
\end{thm}

Fedosov \cite{fed3} has applied this theorem to find a constraint on ``asymptotic operator representations'' of formal deformation quantizations when $\theta = [\frac\omega{2\pi}]$. His notion of an asymptotic operator representation of a formal deformation quantization is almost equivalent to an infinite order strict deformation quantization corresponding to the given formal deformation quantization. The following is a simple adaptation of Fedosov's result.
\begin{thm}
\label{Asymptotic}
Let $\M$ be a compact symplectic manifold and $(I,A,Q)$ an infinite order strict deformation quantization of $\C^\infty(\M)$. Let $\theta$ and $\Tr$ be the characteristic class and trace of the corresponding formal deformation quantization. Suppose that for each $\hbar\neq0\in I$, $A_\hbar$ is represented on a finite-dimensional Hilbert space and that the operator trace $\tr$ in those representations is related to the formal trace by, for any $f\in\C^\infty(\M)$,
\beq
\label{asymptotic_trace}
\tr Q_\hbar(f) \sim \Tr f .
\eeq
Let $c_1(\omega)$ be the first Chern class of the holomorphic tangent bundle determined by any almost complex structure compatible with the symplectic form. Then
\[
\theta + \tfrac12 c_1(\omega) \in \hbar^{-1} H^2(\M)[[\hbar]]
\]
is the asymptotic expansion of a map from $I\smallsetminus \{0\}$ to integral \deRham\ cohomology.
\end{thm}
\begin{proof}
If $e_0\in\Mat_m[\C^\infty(\M)]$ is any idempotent, then there exists \cite[Lem.~5.3]{haw11} an idempotent section $e=e^2\in \Mat_m[\Gamma(I,A)]$ such that $e(0)=e_0$ and which has an asymptotic expansion $e_\hbar\in\C^\infty(\M)[[\hbar]]$:
\[
Q_\hbar[e(\hbar)] \sim Q_\hbar(e_\hbar) .
\]
The condition \eqref{asymptotic_trace} implies that
\[
\tr e(\hbar) \sim \Tr e_\hbar .
\]

The matrix $e_\hbar$ is automatically an idempotent with $e_\hbar\equiv e_0\mod \hbar$, so Theorem~\ref{Algebraic_index} applies and tells us that
\[
\rk e(\hbar) = \tr e(\hbar) \sim \int_M \ch e_0 \wedge e^\theta\wedge \hat A(TM) .
\]
The left side is obviously integer-valued.
%therefore the right side is asymptotically integral.

Let $J$ be an almost complex structure compatible with $\omega$. Let $T_J\M$ be the corresponding holomorphic tangent bundle, so that $c_1(T_JM)=c_1(\omega)$. The  $\hat A$ class can be factorized as $\hat A(TM) = e^{\frac12 c_1(\omega)}\wedge \td(T_JM)$, so for any idempotent $e_0$,
\[
\int_M \ch e_0 \wedge e^\theta\wedge \hat A(TM) = \int_M \ch e_0 \wedge e^{\theta+\frac12 c_1(\omega)}\wedge \td(\omega)
\]
is asymptotically integral.

The bundle $\Wedge^*T_J^*\M$ is a spinor bundle and defines a $\mathrm{Spin}^c$-structure on $\M$, which defines an orientation class $\varepsilon\in K_0(\M)$.
By the Atiyah-Singer index theorem, $\int_M \dots \wedge \td(\omega)$ is $\ch \varepsilon$, the Chern character of $\varepsilon$.

%Since $e_0$ is arbitrary, the possible cap products of $[e_0]\in K^0(\M)$ with this $K$-homology orientation span the $K$-homology $K_0(\M)$.

%A class in $K^0_\co(\M)$ pairs integrally with any element of $K_0(\M)$ if and only if it is actually integral.

The Picard group (of complex line bundles) $\Pic(\M)$ is a multiplicative subgroup of the ring $K^0(\M)$. It can also be identified with $H^2(\M;\Z)$. Taking the Chern character is equivalent to exponentiating; i.e., there is a commutative diagram:
\[
\begin{CD}
H^2(\M;\Z)  @=\Pic(\M)\\
@V{\exp}VV @VVV \\
H^{\mathrm{ev}}(\M) @<{\ch}<< K^0(\M) .
\end{CD}
\]
So, for any $\sigma \in H^2(\M) = \Pic(\M)\otimes \R$,
\[
\begin{split}
\int_M \ch e_0 \wedge e^\sigma \wedge \td (T_JM) &= \langle \ch e_0 \wedge e^\sigma,\ch \varepsilon\rangle = \langle \ch e_0\wedge \ch \sigma,\ch \varepsilon\rangle \\
&= \langle [e_0]\cup \sigma,\varepsilon\rangle = \langle \sigma,[e_0]\cap \varepsilon\rangle .
\end{split}
\]

Since $\varepsilon\in K_0(\M)$ is an orientation, by Poincar\'e duality, any class in $K_0(\M)$ is the cap product of $\varepsilon$ with a class in $K^0(\M)$, and any class in $K^0(\M)$ is a formal difference of projections. Furthermore, $\sigma \in H^2(\M) = \Pic(\M)\otimes \R$ is integral if and only if it pairs integrally with any class in $K_0(\M)$.

Now, let $\sigma$ be the partial sum of $\theta + \frac12 c_1(\omega)$ up to order $\hbar^N$. This shows that $\Norm{\langle \sigma,[e_0]\cap \varepsilon\rangle} = \Or(\hbar^{N+1})$. (The double bars again denote the distance from the integers.) Since this is true for any $e_0$, this implies that the nonintegral part of $\sigma$ is of order $\Or(\hbar^{N+1})$.
\end{proof}

This is not a completely general result, because of the assumption that each $A_{\hbar\neq 0}$ is represented on a finite-dimensional Hilbert space. This is not true for the example of the noncommutative torus.

To see how this relates to rational approximations, again consider the simplest case, when $H^2(\M;\Z)\cong \Z^2$ and suppose that $[\frac\omega{2\pi}]=(\alpha,1)$.
Theorem \ref{Asymptotic} tells us that
\[
\theta + c_1(\omega) = (\alpha,1)\hbar^{-1}+\dots
\]
is the asymptotic expansion of some map $(r,s) : I\smallsetminus\{0\} \to \Z^2$. The second component of $\theta + c_1(\omega)$ is a formal Laurent series consisting of $\hbar^{-1}$ and nonnegative powers of $\hbar$. This can be functionally inverted and inserted into the first component. That is, $r$ can be written as a formal power series in $s$. This power series is the asymptotic expansion of $r$ in terms of $s$, so the range of $(r,s)$ is an infinite order rational approximation to $\alpha$.

\section{Examples}
\label{Examples}
\subsection{The Rational Case}
Suppose that $\alpha = a/b$ where $a,b\in\Z$ and $\gcd(a,b)=1$.

The obvious infinite order rational approximation to this is $\{(ka,kb)\mid k\in\Z\}$. This can also be modified by adding integer constants. In fact, that is all that we can do.

\begin{prop}
\label{Rational uniqueness}
Let $\alpha = a/b$ with $a,b\in\Z$. % and $\gcd(a,b)=1$. 
If $\RR$ is a first order rational approximation to $\alpha$, then there exists $d\in\Z$ such that
\beq
\label{linear_approximation}
br = as + d
\eeq
for all but finitely many $(r,s)\in\RR$. Moreover, $\RR$ is an infinite order rational approximation.
\end{prop}
\begin{proof}
Being a first order rational approximation means that there exists a real number $\gamma_1\in\R$ such that for all $(r,s)\in\RR$,
\[
\frac{r}{s} = \frac{a}{b} + \frac{\gamma_1}{s} + o(s^{-1}) .
\]
Multiplying by $s$ and $b$ gives,
\beq
\label{multiplied_out}
b r = a s + \gamma_1 b + o(1) .
\eeq
Since the first two terms are integers, this means that
\[
\Norm{\gamma_1 b} = o(1) ,
\]
where $\Norm{\;\cdot\;}$ again denotes the distance from $\Z$. However, since the left hand side is a constant, this shows that $\Norm{\gamma_1 b}=0$, that is $d:= \gamma_1 b \in \Z$.

Inserting this back into eq.~\eqref{multiplied_out} gives that $b r = a s + d + o(1)$, but since the first 3 terms are integers, the error $o(1)$ must be $0$ for $\abs r + \abs s$ sufficiently large. This gives eq.~\eqref{linear_approximation}.

Since
\[
\frac{r}{s} = \frac{a}{b} + \frac{d}{b s} ,
\]
this satisfies the definition of an infinite order rational approximation, with coefficients $\gamma_j=0$ for $j\geq2$.
\end{proof}

\subsection{Quadratic irrationals}
First consider the ``golden ratio'' $\phi := \frac{1+\sqrt5}2$. Its continued fraction expansion is simply $\phi = [1;1,1,\dots]$. The partial quotients are  $a_n=1$ for all $n$, so eq.~\eqref{cffact1} shows that the principal convergents are given by Fibonacci numbers,
\[
p_n = F_{n+2} \quad\text{and}\quad q_n = F_{n+1} ,
\]
which are defined recursively by $F_0=0$, $F_1=1$, and
\[
F_n = F_{n-1} + F_{n-2} ,
\]
or explicitly as
\beq
\label{fib explicit}
F_n = \tfrac1{\sqrt5}[\phi^n-(-\phi)^{-n}] .
\eeq

The golden ratio is a root of the polynomial equation $\phi^2-\phi-1=0$, so consider the related homogeneous polynomial $r^2-rs-s^2$. Equation~\eqref{fib explicit} shows that consecutive Fibonacci numbers satisfy
\[
F_{n+1}^2-F_{n+1}F_n - F_n^2 = (F_{n+1}-\phi F_n)(F_{n+1}+\phi^{-1}F_n)= (-1)^n .
\]
This shows that,
\[
\Abs{\frac{F_{n+1}}{F_n} - \phi} \leq \frac{1}{\phi F_n^2} ,
\]
for $n\geq1$, so the set of principal convergents gives a first order rational approximation.
However, it is not a second order rational approximation, because
\[
\frac{F_{n+1}}{F_n} - \phi \approx \frac{(-1)^n}{\sqrt5 F_n^2}
\]
is not a nice function of $F_n$.

Instead, this alternates between two nice functions of $F_n$, so let
\[
\RR := \{ (F_{2k+1},F_{2k}) \mid k\in\N\} .
\]
These pairs of numbers are generated by starting from $(2,1)$ and applying the recursion 
\beq
\label{recursion}
(r,s)\mapsto (2r+s,r+s) .
\eeq
%Define $r_k = F_{2k+1}$ and $s_k=F_{2k}$. This can be defined recursively by, $r_1=2$, $s_1=1$, and
%\beq
%\label{recursion}
%\begin{split}
%r_{k+1} &= 2r_k+s_k\\
%s_{k+1} &= r_k+s_k .
%\end{split}
%\eeq
These satisfy
\[
r^2-rs-s^2 = 1 ,
\]
so this set is just
\[
\RR = \{(r,s)\in\N^2 \mid r^2-rs-s^2=1\} .
\]
In this case, $r$ is an algebraic function of $s$,
\[
r = \frac{s + \sqrt{5s^2+4}}2 = \frac{1 + \sqrt{5+4s^{-2}}}2 s.
\]
For $s>\frac{\sqrt5}2$, this is given exactly by a Laurent series,
\[
r = \phi s + \frac{\sqrt5}2\sum_{j=1}^\infty \frac{(-\frac85)^j}{j! (2j-1)!!} s^{1-2j} ,
\]
therefore this $\RR$ is an infinite order rational approximation to $\phi$.

The growth of the denominators $F_{2k}$ as $k\to\infty$ is extremely different from the rational case. Instead of growing linearly with $k$, they grow exponentially: $F_{2k} \approx \frac1{\sqrt5} \phi^{2k}$.

For any $d\neq0\in\Z$, there exist natural numbers $r,s\in\N$ with $r^2-rs-s^2=d$. The recursion \eqref{recursion} preserves this polynomial, and therefore the set
\[
\RR_d := \{(r,s)\in\N^2 \mid r^2-rs-s^2=d\}
\]
is infinite, and for the same reasons, it is an infinite order rational approximation to $\phi$.

In general, the behavior for quadratic irrationals is similar.
\begin{thm}
\label{Quadratic uniqueness}
If $\RR$ is any second order rational approximation to a quadratic irrational $\alpha$ with $\gamma_1=0$, then $\RR$ is actually an infinite order rational approximation, and there exist $a,b,c,d\in\Z$ such that
\[
a r^2+b rs + c s^2 = d
\]
for all but finitely many $(r,s)\in\RR$.
\end{thm}
\begin{proof}
Being a quadratic irrational means that there exist $a,b,c\in\Z$ such that $0 = a\alpha^2+b\alpha + c$. Inserting $(r,s)\in\RR$ into the corresponding homogeneous polynomial gives
\begin{align*}
a r^2 + b rs + c s^2 & = a (\alpha s + \gamma_2 s^{-1}+o(s^{-1}))^2 + b (\alpha s + \gamma_2 s^{-1}+o(s^{-1}))s + c s^2\\
&= (2a \alpha+b)\gamma_2 + o(1) .
\end{align*}
Since the left side is always an integer, this implies that $\Norm{(2a \alpha+b)\gamma_2} = o(1)$, but since this is a constant, that implies that $d :=(2a \alpha+b)\gamma_2 \in \Z$. Now, the integers $a r^2 + b rs + c s^2 - d$ converge to $0$ as $\abs r + \abs s\to\infty$, which means that they must almost all equal $0$.

This shows in particular that for large enough $s$
\[
r = \frac{-bs\pm\sqrt{(b^2-4ac)s^2+4ad}}{2a},
\]
where the sign is chosen such that $\alpha = \frac{-b\pm\sqrt{b^2-4ac}}{2a}$. Then we have that
\[\frac{r}{s} = \frac{-b\pm\sqrt{b^2-4ac}\sqrt{1+\frac{4ad}{s}}}{2a},\]
and expanding $\sqrt{1+\frac{4ad}{s}}$ as a power series in $4ad/s$ thus exhibits that $\RR$ is an infinite order rational approximation to $\alpha$.
\end{proof}

\begin{thm}
Every quadratic irrational real number $\alpha$ has an infinite order approximation with denominators which grow at most exponentially. 
\end{thm}
\begin{proof}
A quadratic irrational real number has an eventually periodic continued fraction expansion (see \cite[Theorem III.1.2]{RockettSzusz1992}). Therefore we write
\[
\alpha=[0;a_1,\ldots ,a_K,\overline{a_{K+1},\ldots ,a_{K+L}}],
\]
for some integers $K$ and $L$, and let
\[
\gamma_2:=\frac{(-1)^{K+1}}{\zeta_{K+1}+[0;\overline{a_{K+L},\ldots ,a_{K+1}}]}.
\]
Now by \eqref{cffact3} and \eqref{cffact4}, for any positive integer $k$,
\begin{multline}
\left|\alpha-\frac{p_{K+2kL}}{q_{K+2kL}}+\frac{\gamma_2}{q_{K+2kL}^2}\right|\\
=\left|\frac{(-1)^{K+2kL}}{q_{K+2kL}^2(\zeta_{K+2kL+1}+\xi_{K+2kL})}+\frac{(-1)^{K+1}}{q_{K+2kL}^2(\zeta_{K+1}+[0;\overline{a_{K+L},\ldots ,a_{K+1}}])}\right|\\
=\frac{1}{q_{K+2kL}^2}\left|\frac{1}{\zeta_{K+1}+[0;a_{K+2kL},\ldots ,a_1]}-\frac{1}{\zeta_{K+1}+[0;\overline{a_{K+L},\ldots ,a_{K+1}}]}\right|\label{quad est1}
\end{multline}
Now since
\[\lim_{k\rar\infty}[0;a_{K+2kL},\ldots ,a_1]=[0;\overline{a_{K+L},\ldots ,a_{K+1}}],\]
this proves that the quantity in (\ref{quad est1}) is $o(q_{K+2kL}^{-2})$ as $k\rar\infty$. In other words the set 
\[
\RR=\{(p_{K+2kL},q_{K+2kL})\mid k\in\N\}
\] 
is an order $2$ rational approximation to $\alpha$, with $\gamma_1=0$ and $\gamma_2$ as above. By Theorem~\ref{Quadratic uniqueness}, it is actually an infinite order approximation.

Since the continued fraction for $\alpha$ is periodic, there is a constant $M$ such that $a_n\leq M$ for all $n$. By \eqref{cffact1} we have that
\[
q_{K+2kL}= \Or( M^{K+2kL}),
\]
which verifies that the denominators in $\RR$ grow no more than exponentially.
\end{proof}
Note that in the proof of this theorem we used (\ref{cffact3}), which corresponded in our situation with taking $\gamma=0$ in Corollary \ref{nalph-gam formula2}. It may be the case that using the full generality of Corollary \ref{nalph-gam formula2} could produce infinite order approximations to other real numbers that grow more slowly than those constructed in the proof of Theorem \ref{arbitrary approx thm}.

\section{Further Questions}
\label{Conclusions}
This was only a beginning at investigating this topic. Although we have shown that infinite order rational approximation always exist, there are other basic questions that remain to be answered.

\subsection{Growth}
Given a number $\alpha$, how fast do its rational approximations grow? That is, if an order $N$ rational approximation to $\alpha$ is arranged into a sequence, then how fast must the numbers grow?

We have seen that a rational number has infinite order approximations that grow linearly and a quadratic irrational has approximations that grow exponentially. Does the existence of a linearly or exponentially growing approximation imply that $\alpha$ is rational or quadratic?

\subsection{Uniqueness}
To what extent are the expansion coefficients $\gamma_1,\gamma_2,\dots$ restricted by $\alpha$? 

Corollary~\ref{Existence} shows that there always exists at least one approximation with $\gamma_j=0$ for $j\geq2$. Proposition~\ref{Rational uniqueness} shows that, for $\alpha$ rational, the expansion must be of this form, and $\gamma_1$ is greatly restricted. Theorem~\ref{Quadratic uniqueness} shows that if $\alpha$ is quadratic and $\gamma_1=0$, then the other expansion coefficients are determined by a single integer.

\subsection{Generalization}
Proposition \ref{Toeplitz2} and Theorem~\ref{Asymptotic} actually motivate a more general definition. Rather than considering only a single real number, we could take a point  $\alpha\in \R P^n$ and look for a sequence in $\Z^{n+1}$ that converges modulo $\R^\times$ to $\alpha$. All of the questions about rational approximations can be asked again in this more general context.

\end{document}